\documentclass[]{article}
\usepackage{graphicx}
\usepackage{caption}
\usepackage{algorithmic}
\usepackage{algorithm}
\usepackage{amsmath}
\usepackage{amssymb}
\usepackage{amsfonts}
\usepackage{amsthm}
\usepackage{fullpage}
\usepackage{url}
\usepackage{color}
\usepackage{float}
\usepackage{mathabx} % for widecheck symbol
\usepackage{subfig}

\def\N{{\mathbb{N}}}
\def\R{{\mathbb{R}}}
\def\Q{{\mathbb{Q}}}
\def\Z{{\mathbb{Z}}}
\def\dt{{\textrm{d}t}}

\def\d{{\textrm{d}}}

\def\IMF{{\textrm{IMF}}}

\def\iDTFT{{\textrm{iDTFT}}}
\def\DTFT{{\textrm{DTFT}}}

\newtheorem{theorem}{Theorem}
\newtheorem{proposition}{Proposition}
\newtheorem{corollary}{Corollary}
\newtheorem{definition}{Definition}
\newtheorem{problem}{Problem}

\title{One or two frequencies? The Iterative Filtering answers}

\author{Antonio Cicone\thanks{DISIM, Universit\`a degli Studi dell'Aquila, L'Aquila, Italy, and Istituto di Astrofisica e Planetologia Spaziali, INAF, Roma, Italy, and
Istituto Nazionale di Geofisica e Vulcanologia, Roma, Italy. antonio.cicone@univaq.it}, 	
Stefano Serra-Capizzano \thanks{Department of Science and High Technology, University of Insubria, Como, Italy, and Division of Scientific Computing, Department of Information Technology, Uppsala University, Uppsala, Sweden. stefano.serrac@uninsubria.it, stefano.serra@it.uu.se},
Haomin Zhou\thanks{School of Mathematics, Georgia Institute of Technology, Atlanta, GA, U.S.A. hmzhou@gatech.edu}}

\begin{document}

\maketitle

\begin{abstract}
The Iterative Filtering method is a technique aimed at the decomposition of non-stationary and non-linear signals into simple oscillatory components. This method, proposed a decade ago as an alternative technique to the Empirical Mode Decomposition, has been used extensively in many applied fields of research and studied, from a mathematical point of view, in several papers published in the last few years. However, even if its convergence and stability are now established both in the continuous and discrete setting, it is still an open problem to understand up to what extent this approach can separate two close-by frequencies contained in a signal.

In this paper, following the studies conducted on the Empirical Mode Decomposition and the Synchrosqueezing methods, we analyze in detail the abilities of the Iterative Filtering algorithm in extracting two stationary frequencies from a given signal. In particular, after reviewing the Iterative Filtering technique and its known properties, we present new theoretical results and numerical evidence showing the ability of this technique in addressing the fundamental ``one or two frequencies'' question.
\end{abstract}

\section{Introduction}\label{sec:Intro}
Signals consisting of multiple oscillatory components, that are changing their amplitude and frequency while propagating in time, are generated in a great variety of experiments. The analysis of such signals requires a good methodological approach and mathematical apparatus, which allows to find out the main features of the signal by a signal transformation. The most general methods used for the study of stationary signals are wavelet and Fourier transforms analysis. However these methods proved to be limited in identifying the main features when the signals are of non-stationary type. The representation of non-stationary signals in both time and frequency domain is very important for signal analysis in various applications such as, for instance, speech signal analysis and processing, biomedical signal processing, telecommunication engineering, mechanical engineering, seismic signal processing, and many more. Indeed, the signals studied in the previously mentioned areas are non-stationary: their amplitude and frequency parameters vary with respect to time \cite{B.Boashash2003}.

Given the need to study and analyze non-stationary signals, various techniques have been developed over the decades to improve the behavior and performance of wavelet and Fourier transform based methods in dealing with such kind of signals. We can think, for instance, to the reassignment method \cite{auger1995improving}, and the Synchrosqueezed wavelet transform \cite{daubechies2011synchrosqueezed}.

In 1998 a completely new kind of method was introduced by Huang et al. in \cite{huang1998empirical}, the so called Empirical mode decomposition (EMD). This is an iterative method which allows to analyze non-stationary signals stemming from nonlinear systems. By using the EMD we decompose a given signal into simple components. At the end, the original signal can be expressed as a sum of amplitude and frequency modulated (AM-FM) functions called ``intrinsic mode functions'' (IMFs), plus a final monotonic trend. Over the years, EMD proved to be a powerful method which had a broad impact in many applied field of research, see e.g. \cite{echeverria2001application,boudraa2007noise,lei2013review,stallone2020new} and references there in. However, its convergence analysis is still an open problem, e.g. \cite{huang2009convergence,huang2014introduction,ge2018theoretical}.

Iterative Filtering (IF) method is an iterative algorithm alternative to the EMD, that has been introduced by Lin et al. in 2009 \cite{lin2009iterative}. It uses the same algorithm framework as the original EMD, but the moving average of a signal $s(x),\, x\in R$, is derived by the convolution of $s(x)$ with low pass filters. To construct smooth low pass filters, we can employ a Fokker-Planck equation, a second order partial differential equation. The derived filters have compact support and we call them FP filters. The IF, even if published only few years ago, already had an impact in many applied fields of research, e.g. \cite{coifman2017,sharma2017,mitiche2018,li2018entropy,stallone2020new}. Furthermore, this algorithm has been extensively studied in recent years, and its mathematical convergence and stability can be guaranteed a priori both in the continuous and the discrete setting \cite{cicone2016adaptive,cicone2019Direct,cicone2021numerical}. Nevertheless, its ability in separating two close-by frequencies has not yet been studied systematically.

In this work, taking a cue from the study conducted on the EMD first \cite{flandrin2007oneOrTwo}, and on the Synchrosqueezed wavelet transform afterward \cite{wu2011oneOrTwo}, we analyze what kind of separation can (or cannot) be achieved for two-tones composite signals when using the IF method.

The outline of the paper is as follows. In Section 2 we introduce some basic definitions and we present a brief explanation regarding the Iterative Filtering (IF) method in both continuous and discrete settings. We provide theoretical results related to convergence analysis of this technique, when applied to the separation of two stationary frequencies, and show some numerical results. Section 3 concludes the paper.

\section{IF basics}\label{sec:IF}

Let us start by quickly reviewing the algorithm and its features. For a more detailed presentation we refer the interested reader to the paper \cite{cicone2021numerical}.

One of the fundamental ingredients needed in the following is the definition of window/filter function.
\begin{definition}\label{def:window}
A filter/window $w$ is a nonnegative and even continuous function compactly supported  in $[-L,\ L]$, $L>0$, such that $\int_\R w(z)\d z=\int_{-L}^{L} w(z)\d z=1$. $L$ is called \emph{filter length} and represents the half support length of $w$
\end{definition}
Furthermore, from now on, the notation $\widehat{f}(\xi)$ will be used to represent the Fourier transform of a function $f$ computed at the frequency $\xi$.
The IF pseudocode is given in Algorithm \ref{algo:IF}
\begin{algorithm}
\caption{\textbf{Iterative Filtering} IMF = IF$(s)$}\label{algo:IF}
\begin{algorithmic}
\STATE IMF = $\left\{\right\}$
\WHILE{the number of extrema of $s$ $\geq 2$}
\STATE $s_1 = s$
\WHILE{the stopping criterion is not satisfied}
\STATE  compute the filter length $L_m$ for $s_{m}(x)$
\STATE  $s_{m+1}(x) = s_{m}(x) -\int_{-L_m}^{L_m} s_m(x+t)w_m(t)\dt$
\STATE  $m = m+1$
\ENDWHILE
\STATE IMF = IMF$\,\cup\,  \{ s_{m}\}$
\STATE $s=s-s_{m}$
\ENDWHILE
\STATE IMF = IMF$\,\cup\,  \{ s\}$
\end{algorithmic}
\end{algorithm}
where $w_m(t)$ is a given filter, like one of the Fokker-Plank filters proposed in \cite{cicone2016adaptive}.

The algorithm contains two loops: an inner and an outer loop, the second and first while loop in the pseudocode, respectively. The former captures a single IMF, while the latter produces all the IMFs embedded in a signal.

Assuming $s_1=s$, the key step in the algorithm is the moving average computation of $s_m$ performed as
\begin{equation}\label{eq:Mov_Average}
\mathcal{L}_m(s_m)(x)=\int_{-L_m}^{L_m} s_m(x+t)w_m(t)\dt,
\end{equation}
which represents the convolution of the signal itself with the window/filter $w_m(t)$.

The moving average is then subtracted from $s_m$ to obtain its fluctuation part
\begin{equation}\label{eq:flactuations}
\mathcal{M}_{m}(s_m)= s_m-\mathcal{L}_m(s_m)=s_{m+1}.
\end{equation}

$\textrm{IMF}_1$ is computed by repeating iteratively this procedure on the signal $s_m$, $m\in\N$, until a stopping criterion is satisfied \cite{cicone2016adaptive}.

In order to produce the subsequent IMFs we apply the same procedure to the remainder signal $r=s-\sum_{i=1}^{k}\textrm{IMF}_i$, where $k$ is the number of previously extracted IMFs.

The algorithm stops when $r$ becomes a trend signal, meaning it has at most one local extremum.
Regarding the filter length $L_m$, its selection is performed at the first step of each inner loop and then it is kept constant throughout the entire inner loop. Hence $L_m=L_1=L$ for every $m\geq 1$.

From now on we assume that the algorithm is making always an exact filter length selection. However it is important to remind that the mask length selection is a problem per se. In \cite{lin2009iterative,cicone2016adaptive,cicone2021numerical} the authors suggest different approaches to its computation. Furthermore, we remind that the computation of the filter length has to be based on the signal itself, in order to make the method nonlinear \cite{cicone2021numerical}.

In this work we want to study the ability of the IF algorithm to resolve frequencies of a given signal especially when they are close each other. In order to address this question Flandrin and Rilling in  \cite{flandrin2007oneOrTwo} propose to study the signal
\begin{equation}\label{eq:signal}
s(x, a, f) = \cos(2\pi x)+a \cos(2\pi f x+\phi) \qquad x\in\R \textrm{ and } f \in (0, 1).
\end{equation}
The term $\cos( 2\pi x)$ is referred to in the following as the high frequency component (HF) and the second term as the low frequency one (LF).

\subsection{The continuous setting}

The goal of this section is to study the ability of the IF method to decompose properly the signal \eqref{eq:signal} into two pure tones.

We start by recalling the following theorem which regards the convergence analysis of the Iterative Filtering inner loop.

\begin{theorem}[Convergence of the Iterative Filtering method Theorem\cite{cicone2016adaptive,huang2009convergence}]\label{thm:theo_1}
Given the filter function $w(t)$ in $L^2\left([-L,L]\right)$, and let $s(x)\in L^2(\mathbb{R})$. \newline
If $|1- \widehat{w}(\xi)| < 1 $ or $\widehat{w}(\xi)=0$,

Then
$\{\mathcal{M}^m(s)\}$ converges and
\begin{equation}\label{eq:IMF_cont}
\textrm{IMF}_1 = \lim\limits_{m\rightarrow \infty}{\mathcal{M}^m(s)(x)}= \int_{-\infty}^{\infty} \widehat{s}(\xi) \chi_{\{\widehat{w}(\xi)=0 \}}
 e^{2\pi i \xi x} \textrm{d}\xi.
\end{equation}
\end{theorem}

We observe that given $h:[-\frac{L}{4},\frac{L}{4}]\rightarrow\R$, $z\mapsto h(z)$, nonnegative, symmetric, with $\int_\R h(z)\d z=\int_{-\frac{L}{4}}^{\frac{L}{4}} h(z)\d z=1$, it is sufficient to construct the window $w$ as convolution of $h$ with itself to ensure the method convergence to the limit function \eqref{eq:IMF_cont}, which depends only on the shape of the filter function itself and the support length selected by the method \cite{cicone2016adaptive,cicone2019spectral}.

In general we can assume that the filter function $w_L$ supported on $[-L,\ L]$ is defined as some scaling of an a priori fixed filter shape $w:[-1,1]\rightarrow\R$.

For simplicity, from now on, we consider the linear scaling
\begin{equation}\label{eq:w_m_linear}
w_L(x)=\frac{1}{L}w\left(\frac{x}{L}\right).
\end{equation}

We are now ready to state the following theorem.
\begin{theorem}\label{thm:IF_separation}
Given the signal $s$ defined in \eqref{eq:signal}, assuming that the IF method is selecting accurately the mask length of the doubly convolved filter  $w_L$ so that the lowest positive frequency of the $\widehat{w_L}$ function with zero value corresponds to frequency $1$,

Then IF algorithm in the continuous setting can always resolve the signal \eqref{eq:signal} into two components no matter how close the frequency $f$ of the LF component is to 1.
\end{theorem}

\begin{proof}
We start by recalling that, if we define $\widehat{w}(\xi)=\int_{-\infty}^{+\infty} w(x) e^{-i\xi x 2 \pi} \d x$, then
\begin{equation}\label{eq:fft_w_L}
\widehat{w_L}(\xi)=\int_{-\infty}^{+\infty} \frac{1}{L}w\left(\frac{x}{L}\right) e^{-i\xi \frac{x}{L} L 2 \pi} \d x=\widehat{w}(L\xi).
\end{equation}
Therefore, if $\xi_0$ is a root of $\widehat{w}(\xi)=0$, then $\frac{\xi_0}{L}$ is a root of $\widehat{w_L}(\xi)=0$ because $\widehat{w_L}\left(\frac{\xi_0}{L}\right)=\widehat{w}\left(L\frac{\xi_0}{L}\right)=\widehat{w}(\xi_0)=0$.

Furthermore, we assume that the IF method selects properly the mask length $L$ so that the lowest positive frequency of the $\widehat{w_L}$ function with zero value corresponds to frequency $1$.
Hence, based on \eqref{eq:IMF_cont}, the only component that is captured in the first IMF by the IF algorithm correspond to the HF component contained in $s$.

Finally, since $w$ are compactly supported functions, $\widehat{w}$ are defined on $\R$ and they have zeros which are isolated points.

Then, from the previous observations, it follows that IF algorithm is able to separate exactly the two components of $s$, no matter how close the frequency $f$ of the LF component is to 1.
\end{proof}

\subsection{The continuous setting with a stopping criterion}
It is reasonable to introduce some stopping criterion \cite{cicone2016adaptive,cicone2021numerical} in order to achieve a decomposition in finite time.

In particular we assume to use the following stopping criterion.
\begin{problem}\label{stopping_criterion}
Fixed $\delta > 0$ we want to find the minimum value $N_0\in\N$ such that \[\|\mathcal{M}^N(s)-\mathcal{M}^{N+1}(s)\|_{L^2}<\delta \qquad \forall N\geq N_0.\]
\end{problem}

If we do so, Algorithm \ref{algo:IF} converges in finite steps to an IMF whose explicit form is given in the following theorem.

\begin{theorem}[Convergence of the Iterative Filtering method with stopping criterion \cite{cicone2021numerical}]\label{thm:IF_inner_conv_stopping}
Given $s\in L^2(\R)$ and $w$ obtained as the convolution $h\ast h$, where $h$ is a filter/window, and fixed $\delta>0$.

Then, for the minimum $N_0\in\N$ such that $\left\| \mathcal{M}^N(s)(x)-\mathcal{M}^{N+1}(s)(x)\right\|_{L^2}<\delta$, $\forall N\geq N_0$, the first IMF is given by

\begin{equation}\label{eq:IMF_IF_stop}
\textrm{IMF}_1^\textrm{SC}=\mathcal{M}^N(s)(x)=\int_{\R} (1-\widehat{w}(\xi))^N \widehat{s}(\xi) e^{2\pi i \xi x}\d \xi \quad \forall N\geq N_0.
\end{equation}
\end{theorem}

\begin{theorem}\label{thm:IF_separation2}
Given the signal $s$ defined in \eqref{eq:signal}, assuming that the IF method is selecting accurately the mask length of the filter $w_L$ so that the lowest positive frequency of the $\widehat{w_L}$ function with zero value corresponds to the frequency $1$, and fixed $\delta,\ \eta>0$.

Then there exists $N_2\in\N$ such that $\|\mathcal{M}^N(s)(x)-\cos(2\pi  x)\|_{L^2}<\eta$, $\forall N\geq N_2$.
\end{theorem}

\begin{proof}
From \eqref{eq:IMF_IF_stop} it is clear that to obtain a perfect separation we need to have $N\rightarrow\infty$. In fact for any $f\in (0,\ 1)$ and any $N\in\N$, assuming always that the mask length selection is done in the right way so that the lowest positive frequency of the $\widehat{w}$ function with zero value correspond to the frequency $1$, the term $(1-\widehat{w}(\xi))^N \widehat{s}(f)$ it will be always different from zero since $(1-\widehat{w}(\xi))>0$ for any filter produced as convolution of symmetric and nonnegative filters with themselves \cite{cicone2016adaptive}.

It follows from Theorem \ref{thm:IF_inner_conv_stopping} that there exists $N_0\in\N$ such that IF is converging to the HF component contained in the signal \eqref{eq:signal}.
Furthermore, $(1-\widehat{w}(\xi))\rightarrow 0$ as $N\rightarrow\infty$, therefore there exists $N_1\in\N$ such that $\|(1-\widehat{w}(f))^N \widehat{s}(f) e^{2\pi i f x}\|_{L^2}<\eta$, $\forall N\geq N_1$.
Hence for $N_2=\max(N_0,\ N_1)$ it follows the conclusion.
\end{proof}

We point out here that the filter $w$ shape can be chosen in order to make the value of $(1-\widehat{w}(f))^N$ small as we like for a fixed $N\in\N$.

\subsection{The discrete setting}
We consider now the case of discrete signals, which can be either aperiodical or periodical discrete signals supported on $\R$. In the case of periodical signals we can focus on the analysis of a single period and consider it as a vector in $\R^n$, $n\in\N$.

Let us start from the case of aperiodical discrete signals supported on $\R$.

\subsubsection{Aperiodical discrete signals on $\R$}
We assume that the signal \eqref{eq:signal} is sampled at discrete and uniformly spaced points $x_k\in\R$.
For simplicity, and without loosing generality, we assume that
\begin{equation}\label{eq:discrete_signal}
 s(x_k, a, f) = \cos(2\pi x_k)+a \cos(2\pi f x_k+\phi) \qquad \forall k\in\Z \textrm{ and } f \in (\R\backslash\Q)\cap(0, 1),
\end{equation}
where $x_k=k T$, $T\ll 1$ and $k\in\Z$, so that $x$ is sampled with the rate of $\textrm{Fs}=\frac{1}{T}=f_s\gg 1$ samples/seconds which allows to capture all its fine details.

The term $\cos(2\pi x_k)$ is referred to in the following as the high frequency component (HF) and the second term as the low frequency one (LF).

We observe that, since $f$ is irrational, the HF and LF components together form an aperiodical signal $s$.

Assuming $s_1=s$, the main step of the IF method becomes
\begin{equation}\label{eq:s_m+1}
s_{m+1}(x_k) \approx s_{m}(x_k)-\!\!\!\!\! \sum_{x_j=x_k-L_m}^{x_k+L_m}\!\!\!\!\! s_m(x_j)w_m(x_k-x_j)\frac{1}{n}, \quad k\in\Z.
\end{equation}

Given the Continuous Fourier Transform (CFT) $\widehat{s}(\xi)$ of the original signal \eqref{eq:signal}, then the Discrete Time Fourier Transform (DTFT) of the uniformly sampled signal \eqref{eq:discrete_signal} and its inverse are equal to
\begin{eqnarray}
%\nonumber to remove numbering (before each equation)
\label{eq:DTFT1} \widehat{s}_{1/T}(\xi) &=& \sum_{h=-\infty}^{\infty} \widehat{s}\left(\xi-\frac{h}{T}\right),\\
\label{eq:DTFT2} s(x_k) &=& T \int_{\frac{1}{T}} \widehat{s}_{1/T}(\xi) e^{i2\pi\xi k T}\d \xi,
\end{eqnarray}
for any $\xi\in\R$ and $k\in\Z$.

The integer $h$ in \eqref{eq:DTFT1} has units of cycles/sample, and $1/T=f_s$ is the sample rate in samples/seconds. So $\widehat{s}_{1/T}(\xi)$ comprises exact copies of $\widehat{s}(\xi)$ that are shifted by multiples of $f_s$ hertz and added together. Assuming we are dealing with a compactly supported filter $w$, for sufficiently large $f_s$ the $h = 0$ term in the DTFT of $w$ can be observed in the region $[-f_s/2, f_s/2]$ with little contribution (aliasing) from the other terms.

We are now ready to state the following.

\begin{proposition}[Convergence of the Iterative Filtering method applied to aperiodical discrete signals on $\R$]\label{pro:DTFT_convergence}
Given the filter function $w(t)$ and an aperiodical discrete signal $s(kT)$, $k\in\Z$. \newline
If $|1- \widehat{w}_{1/T}(\xi)| < 1 $ or $\widehat{w}_{1/T}(\xi)=0$,

Then IF converges and the first IMF equals
\begin{equation}\label{eq:IMF_disc_aper}
\textrm{IMF}_1 = T \int_{\frac{1}{T}} \widehat{s}_{1/T}(\xi) \chi_{\{\widehat{w}_{1/T}(\xi)=0 \}}  e^{i 2\pi \xi k T} \d \xi.
\end{equation}
\end{proposition}

The proof follows directly from the one of Theorem \ref{thm:theo_1} and the properties of the DTFT.

\begin{corollary}\label{cor:IF_resolution_1}
Given the signal $s$ defined in \eqref{eq:discrete_signal}, given a filter $w$ and assuming its DTFT $\widehat{w}_{1/T}$ contains at least one zero value. If we further assume that the IF method is selecting accurately the linear scaling $p$ of the doubly convolved filter such that its DTFT, $\widehat{w_p}_{1/(pT)}$, has its lowest positive frequency with zero value at frequency~$1$,

Then IF algorithm can resolve exactly $s$ into the two components HF and LF.
\end{corollary}

\begin{proof}
From the hypotheses it follows that the filter used in IF algorithm will be $w_p$, whose DTFT, $\widehat{w_p}_{1/(pT)}$, has its lowest positive frequency with zero value at frequency~$1$. Therefore, from Proposition \ref{pro:DTFT_convergence} it follows that the first IMF produced by IF will contain all and only the HF component of the given signal $s$, no matter how close the frequency $f$ of the LF component is to 1.
\end{proof}

From an intuitive point of view it may seem pretty hard to guarantee that the assumptions of this Corollary hold true. In particular, having a filter $w(x)$ which has at least one zero in its DTFT is clearly not common in general. In fact, the DTFT of $w$ is, as suggested in \eqref{eq:DTFT1}, the summation of infinitely many repetitions of a properly shifted CFT of the continuous version of $w$. In order to have a zero in $w$ DTFT, we need the $w$ CFT zero frequencies positions to align in at least in one frequency position when we compute the DTFT. However it could be tricky to build such a filter. Furthermore, at every scaling of the filter, this alignment of the zeros may be lost.

The following theorem provides an easy way to construct a filter which is guaranteed to have an actual zero in the DTFT domain.

\begin{theorem}\label{thm:FilterDTFT}
Given a compactly supported filter $h_p(kT)$, $k\in\Z$, assuming the smallest positive minimum in its DTFT is at frequency $f_1$ and has value $\varepsilon$,

Then the function
\begin{equation}\label{eq:newFilter}
w_p=\iDTFT\left(\left(\DTFT(h_p)-\varepsilon\right)^2\right)
\end{equation}
is a doubly convolved real filter with zero at frequency $f_1$ in the DTFT domain.
\end{theorem}
\begin{proof}
First of all, we observe that from \eqref{eq:DTFT2} it follows that $T \int_{\frac{1}{T}} \varepsilon e^{i2\pi\xi k T}\d \xi = T \varepsilon \delta[kT]$, where $\delta(kT)$ is a Dirac delta function in discrete time for $k\in\Z$.

Hence, by the DTFT properties, from \eqref{eq:newFilter} it follows that $w_p$ is a filter doubly convolved with itself. In particular, $w_p$ is even, nonnegative and compactly supported function such that $\sum_{k=-\infty}^\infty w_p(kT)=1$.
Furthermore, by construction, $w_p$ has a zero at frequency $f_1$.
\end{proof}

\subsubsection{Periodical discrete signals on $\R$}

We consider now the case of vectorial signals defined in $\R^p$, $p\in\N$.

\begin{equation}\label{eq:discrete_signal_Rn}
   s(x_k, a, f) = \cos(2\pi x_k)+a \cos(2\pi f x_k+\phi), \qquad \forall k=1,\ \ldots,\ \frac{n}{T}=p \textrm{ and } f \in \left(\frac{1}{n}, 1\right),
\end{equation}
where $n$ and $\frac{1}{T}\in\N$, $x_k=k T$, $T\ll 1$, so that $x$ is sampled with the rate of $\textrm{Fs}=\frac{1}{T}\gg 1$ samples/sec which allows to capture all its fine details.

We want to decompose the vector $\mathbf{s}=\left[s(x_k)\right]_{k=0}^{\frac{n}{T}-1}$ into two vectorial IMFs. Without loosing generality we can assume that $\|\mathbf{s}\|_2=1$.

We assume that a filter function $w$ has been selected a priori, for instance one of the Fokker-Planck filters \cite{cicone2016adaptive} convolved with itself, and that $w_m$ is computed by linear scaling, as described in \eqref{eq:w_m_linear}, such that its support becomes $2 L_m+1$, $L_m\in\N$.

Assuming $s_1=s$, the main step of the IF method becomes
\begin{equation}\label{eq:s_m+1_2}
s_{m+1}(x_i) \approx s_{m}(x_i)-\!\!\!\!\! \sum_{x_j=x_i-L_m}^{x_i+L_m}\!\!\!\!\! s_m(x_j)w_m(x_i-x_j)\frac{1}{n}, \quad i=0,\ldots,\frac{n}{T}-1,
\end{equation}
where $x$ is sampled at a rate which allows to capture all the fine details of $s$, so that aliasing will not play any role. In particular, given the sampling rate of $\textrm{Fs}=\frac{1}{T}$, from Shannon--Hartley theorem \cite{flandrin1998} we know that the highest frequency that can be properly captured is $\frac{1}{2T}$ Hz. Whereas, given that the signal is assumed to be of length $n$ seconds, the lowest frequency that can be fully represented equals $\frac{1}{n}$ Hz.

Algorithm \ref{algo:IF_discrete} provides the discrete version of IF Algorithm \ref{algo:IF}

\begin{algorithm}
\caption{\textbf{Discrete Iterative Filtering} IMF = DIF$(s)$}\label{algo:IF_discrete}
\begin{algorithmic}
\STATE IMF = $\left\{\right\}$
\WHILE{the number of extrema of $s$ $\geq 2$}
\STATE $s_1 = s$
\WHILE{the stopping criterion is not satisfied}
\STATE compute the function $w_m(\xi)$, whose half support length $L_m$ is based on the signal $\left[s_m(x_i)\right]_{i=0}^{\frac{n}{T}-1}$
\STATE  $s_{m+1}(x_i) = s_{m}(x_i) - \sum_{j=0}^{p-1} s_m(x_j)w_m(|x_i-x_j|) \frac{1}{p},\qquad i= 0,\ldots, p-1$
\STATE  $m = m+1$
\ENDWHILE
\STATE IMF = IMF$\,\cup\,  \{ s_{m}\}$
\STATE $s=s-s_{m}$
\ENDWHILE
\STATE IMF = IMF$\,\cup\,  \{ s\}$
\end{algorithmic}
\end{algorithm}

In matrix form we have
\begin{equation}\label{eq:MatrixForm}
    s_{m+1}=(I-W_m)s_m
\end{equation}
where
\begin{equation}\label{eq:K}
    W_m=\left[w_m(x_i-x_j)\cdot \frac{T}{n}\right]_{i,\ j=0}^{\frac{n}{T}-1}=\left[\frac{w\left(\frac{x_i-x_j}{L_m}\right)}{L_m}\cdot \frac{T}{n}\right]_{i,\ j=0}^{\frac{n}{T}-1}=\left[\frac{w\left(\frac{i-j}{(n-1)L_m}\right)}{L_m}\cdot \frac{T}{n}\right]_{i,\ j=0}^{\frac{n}{T}-1}.
\end{equation}

The first IMF is given by $\textrm{IMF}_1=\lim_{m\rightarrow\infty} (I-W_m)s_m$.

We point out that the matrix $W_m$ depends on the half support length $L_m$ at every step $m$. However, in the implemented code the value $L_m$ is usually computed only in the first iteration of each inner while loop and then kept constant to $L_m=L_1=L$ value in the subsequent steps, so that the matrix $W_m$ is equal to $W_1$ for every $m\in\N$. Thus, the first IMF is given by

\begin{equation}\label{eq:First_IMF_fixed_length}
\textrm{IMF}_1=\lim_{m\rightarrow\infty} (I-W_1)^{m} s.
\end{equation}

\begin{theorem}[Convergence of the Discrete Iterative Filtering \cite{cicone2021numerical}]\label{thm:ExplicitFormulaDiscreteIMF}
Given a signal $s\in\R^p$, assuming that we are considering a doubly convolved filter $w_L$ whose half filter support length $L_m$ is constant and equal to $L$ throughout all the steps of an inner loop, assuming that $L$ is big enough so that the convolution matrix $W_1$ associated with the scaled filter $w_L$ is different from an identity matrix $I$, assuming that $\left\{\lambda_j\right\}_{j=0,\ldots,p-1}$ are the eigenvalues of $W$ such that $k$ of them, $k\in\{0,\ 1,\ldots,\ p-1\}$, are equal to zero,

Then $W_1$ is diagonalizable as $W_1=U D U^T$, where $U$ is a unitary matrix having as columns the eigenvectors $u_p$ of $W_1$, and the first outer loop step of the DIF method converges to
\begin{equation}\label{eq:discreteIMF}
\textrm{IMF}_1=\lim_{m\rightarrow \infty}(I-W_1)^m s=U Z U^Ts
\end{equation}
where  $Z=I-D$ is a diagonal matrix with entries all zero except $k$ elements in the diagonal which are equal to one.
\end{theorem}

We recall here that the eigenvalues of the $W_1$ circulant matrix are given by
\begin{equation}\label{eq:Lambdas}
\lambda_j\ =\ \sum_{q=0}^{p-1} w_L\left(\frac{q}{p-1}\right)e^{-2\pi i j \frac{q}{p}}  \qquad  \qquad j\ =\ 0,\ldots,\ p-1,
\end{equation}
which is equivalent to the Discrete Fourier Transform (DFT) of $w_L$ at frequencies $j\ =\ 0,\ldots,\ p-1$.
The corresponding eigenvectors are
\begin{equation}\label{eq:Eigenvectors}
u_j\ =\ \frac{1}{\sqrt{p}} \left[1,\ e^{-2\pi i j\frac{1}{p}},\ldots,\ e^{-2\pi i j\frac{p-1}{p}}\right]^T, \qquad  \qquad j\ =\ 0,\ldots,\ p-1,
\end{equation}
which form a Fourier basis \cite{cicone2021numerical}. These observations led to a fast implementation of the IF method, the so called Fast Iterative Filtering (FIF) \cite{cicone2021numerical,cicone2019Direct}.

Regarding the filter length $2L+1$, i.e. the filter support length, measured in sample points, we assume in this work that it can only achieve integer values. We recall that it is always possible to go beyond this limit by numerical approximation as described in \cite{cicone2016adaptive} Section 4.
We are now ready to prove the following
\begin{corollary}\label{cor:DIF_resolution}
Given the signal $s\in\R^p$ defined in \eqref{eq:discrete_signal_Rn}, sampled at $\textrm{Fs}=\frac{1}{T}$ sampling rate such that $p=n\frac{1}{T}$, assuming $w_L$ is a doubly convolved filter, whose filter length, measured in sample points, equals $2L+1$, such that the smallest positive zero in the DFT of $w_L$ corresponds to frequency $1$.

Then DIF algorithm can always resolve $s$ into the two components HF and LF as far as the LF component has a frequency $f \leq 1-\frac{1}{n}$.
\end{corollary}

The proof follows directly from Theorem \ref{thm:ExplicitFormulaDiscreteIMF} and the properties of the DFT. In particular, the frequency resolution of the DFT is given by the ratio $\frac{Fs}{p}=\frac{1}{T}\frac{T}{n}=\frac{1}{n}$, which represents the reciprocal of the number of seconds contained in the signal $s$.

We point out that the very same result holds true also for the FIF implementation of DIF method proposed in \cite{cicone2021numerical,cicone2019Direct}.

Another important observation regards the ability of DIF and FIF methods to exactly separate the HF from the LF components. Clearly there are two scenarios: the frequency $f$ of the LF component belongs to the set $\left\{\ \frac{1}{n},\ \frac{2}{n},\ldots,\ 1-\frac{1}{n}\right\}$, i.e. can be exactly represented by the DFT in the Fourier domain, or it does not. In the first case the method can separate exactly the two components. In the second case the algorithm can still identify two components and separate the two with a certain degree of accuracy.

This second scenario includes the case of frequencies $f$ not being rational numbers. Clearly we cannot have a perfect separation, but the algorithm is still able to address the fundamental question ``one or two frequencies?''.

\subsection{The discrete setting with a stopping criterion}

In the implemented algorithm we do not let $m$ to go to infinity. Instead, we use some kind of stopping criterion and discontinue the calculations when a pre-fixed threshold value has been reached \cite{cicone2016adaptive}.

We recall the following known theorem.

\begin{theorem}[Convergence of the Discrete Iterative Filtering with a Stopping Criterion \cite{cicone2021numerical}]\label{thm:DIF_conv_stop}

Given $s\in\R^p$, we consider the convolution matrix $W$, associated with a filter vector $w$ given as a symmetric filter $h$ convolved with itself. Assuming that $W$ has $k$ zero eigenvalues, where $k$ is a number in the set $\in\{0,\ 1,\ldots,\ p-1\}$, and fixed $\delta>0$.

Then, calling $\widetilde{s}=U^T s$, for the minimum $N_0\in\N$ such that it holds true the inequality

\begin{equation*}\label{eq:N0_discrete}
\frac{N_0^{N_0}}{\left(N_0+1\right)^{N_0+1}}<\frac{\delta}{\|\widetilde{s}\|_\infty{\sqrt{p-1-k}}}
\end{equation*}

we have that $\left\| s_{m+1}-s_m\right\|_{2}<\delta \quad \forall m\geq N_0$ and the first IMF is given by
\begin{equation}\label{eq:IMF1_direct}
\overline{\textrm{IMF}}_1=U(I-D)^{N_0} U^T s= U P \left[
\begin{array}{ccccccc}
0 &   &   &   &   &   &   \\
&  (1-\lambda_1 )^{N_0} &   &   &   &   &   \\
&   & \ddots  &   &   &   &   \\
&   &   &  (1-\lambda_{p-1-k} )^{N_0} &   &   &   \\
&   &   &   &  1 &   &   \\
&   &   &   &   &  \ddots &   \\
&   &   &   &   &   & 1  \\
\end{array}
\right] P^T U^T s,
\end{equation}
where $P$ is a permutation matrix which allows to reorder the columns of $U$, which correspond to eigenvectors of $W$, so that the corresponding eigenvalues $\{\lambda_j\}_{j=1,\ldots,\ p-1}$ are in decreasing order.

\end{theorem}

From Theorem \ref{thm:DIF_conv_stop} we deduce the following result.

\begin{corollary}\label{cor:DIF_Stopping_resolution}
	Given the signal $s\in\R^p$ defined in \eqref{eq:discrete_signal_Rn}, sampled at $\textrm{Fs}=\frac{1}{T}$ sampling rate such that $p=n\frac{1}{T}$, assuming the LF component has a frequency $f$ belonging to the set $\left\{\ \frac{1}{n},\ \frac{2}{n},\ldots,\ 1-\frac{1}{n}\right\}$. Assuming also that $w_L$ is a doubly convolved filter, whose filter length, measured in sample points, equals $2L+1$, such that the smallest positive zero in the DFT of $w_L$ corresponds to frequency $1$, and fixed $\eta > 0$ and the stopping criterion $\delta > 0$,
	
	Then, there exists an $\widetilde{N_0}>0$ such that the stopping criterion is satisfied, and $\overline{\textrm{IMF}}_1=U(I-D)^{\widetilde{N_0}} U^T s$ and the ground truth HF component differ in norm less than $\eta$.
\end{corollary}

The proof follows from Theorem \ref{thm:DIF_conv_stop} and the properties of the DFT. In particular, from \eqref{eq:IMF1_direct} it follows that the norm of the difference between the ground truth HF component and $\overline{\textrm{IMF}}_1=U(I-D)^{N_0} U^T s$ is equal to the norm of $(1-\lambda_k )^{N_0}u_k$, where $u_k$ equals the LF component. No matter the value of $\eta>0$ and $\lambda_k\in(0,\,1)$, there will always be a $\widetilde{N_0}>0$ such that the conclusion of the theorem follows.

We point out that similar results can be derived also when the frequency $f\in(0,\, 1)$ of the LF component do not belong to the set $\left\{\frac{1}{n},\ \frac{2}{n},\ldots,\ 1-\frac{1}{n}\right\}$. The calculations become more involved, but also in this case the algorithm can identify two components and separate them with a prefixed degree of accuracy. The results obtained applying the IF algorithm to the case of frequency $f$ irrational are shown in Figures \ref{fig:c1_irr_ideal} and \ref{fig:c1_irr_actual} in the following numerical results section. %In this setting, it does not exist an interval of time for which the frequency 1 and the irrational frequency $f$ are at the same time periodical. Therefore, since we are assuming periodical conditions at the boundaries, the two frequencies in the frequency domain tend to have some overlapping due to a non-perfect reconstruction of the irrational frequency. Therefore when IF method extracts frequency 1 it does unavoidably extract part of the irrational frequency $f$.

\subsection{Numerical results}
We test now the performance of the Iterative Filtering (IF) Matlab code\footnote{\url{www.cicone.com}} in separating two frequencies and we compare the outcome with the EMD and Synchrosqueezing performance, when applied to periodical discrete signals.

We recall the formula proposed by Flandrin et al. in \cite{flandrin2007oneOrTwo} which we use in this work to measure the performance in separating two pure tone components.
\begin{equation}\label{eq:c1}
c_1(a,\ f,\ \phi)\doteq \frac{\|\IMF_1\left(x,\ a,\ f\right) - \cos\left(2\pi x\right)\|_{\textrm{L}^2(T)}}{\|a\cos\left(2\pi f x+\phi\right)\|_{\textrm{L}^2(T)}}.
\end{equation}

In all the following tests we let the IF algorithm iterating for 10 millions times and we set the $\delta$ parameter for the stopping criterion to $10^{-20}$. Same results, in terms of the algorithm decomposition performance, can be obtained by setting the $\delta$ to the standard value of $10^{-3}$ and letting the algorithm to run for just a few iterations. In this work, we opt to put under stress the algorithm and show its ability to converge to a steady solution, which remains unchanged even after a high number of iterations.

We point out that, when we apply a decomposition method to a discrete and compactly supported signal, boundary errors will show up in most cases. In \cite{cicone2019BC} the authors estimated in a rigorous way the propagation of the error inside the decomposition due to the boundaries. In the following tests we apply the approach proposed in \cite{stallone2020new} to mitigate these errors.

We start assuming that the IF method is combined with an ideal way of selecting the mask length $L$. From Figure \ref{fig:c1_ideal}, where the performance of this ideal version of the algorithm are shown, we see that in this case the algorithm can separate, up to machine precision, the two components for any combination of amplitude $a$ and frequency $f<1-\frac{1}{n}$ of the LF component, as predicted by Corollaries  \ref{cor:DIF_resolution} and  \ref{cor:DIF_Stopping_resolution}.

\begin{figure}[h]
\centering
\includegraphics[width=0.5\linewidth]{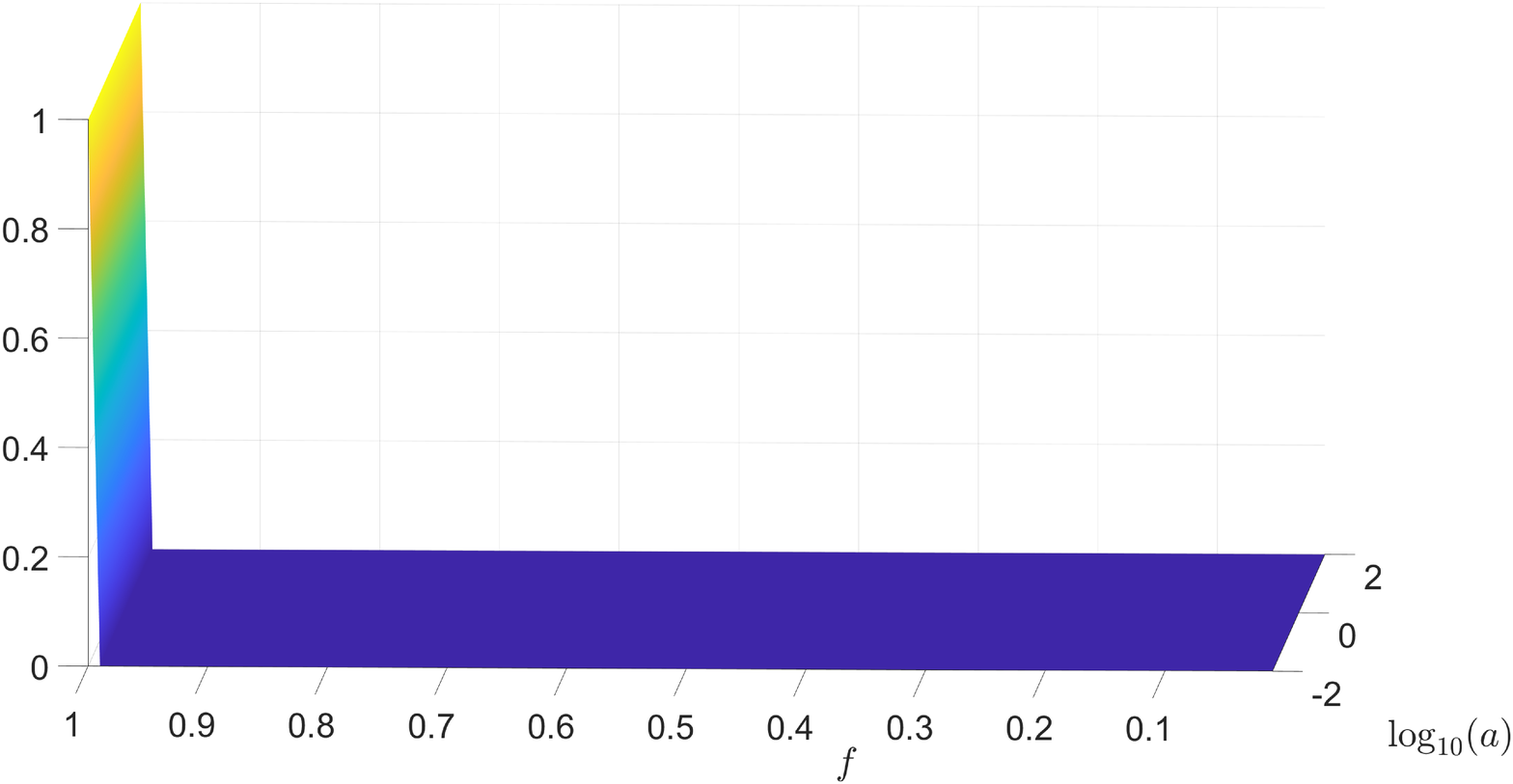}~\includegraphics[width=0.5\linewidth]{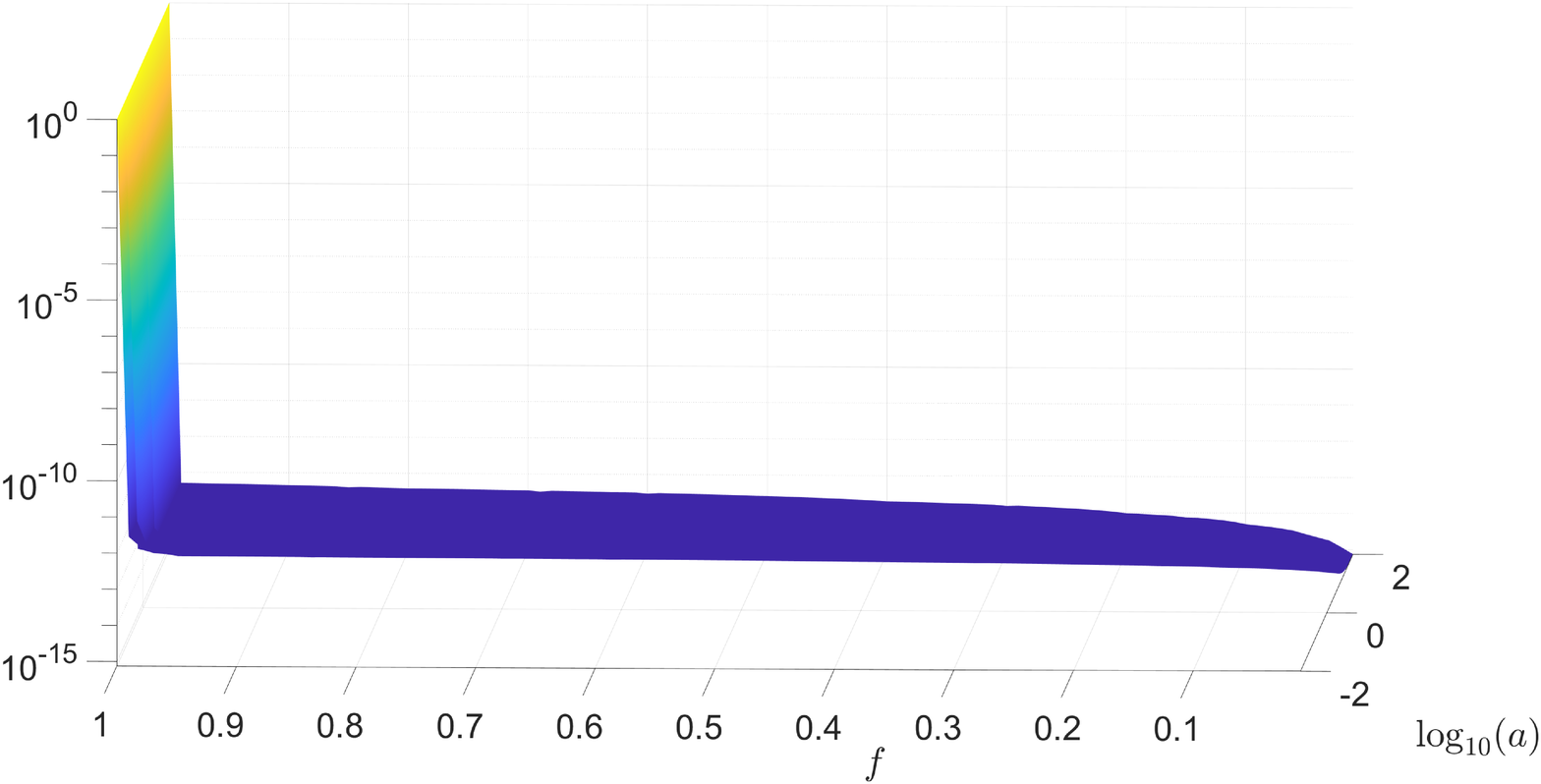}
\caption{Left panel, IF performance in separating two-tones with rational frequencies from the signal \eqref{eq:discrete_signal_Rn}, measured using the criterion \eqref{eq:c1} averaged over $\phi$. We assume that the algorithm selects always the ideal mask length $L$. Right panel, we show the same plot reported on the left, but with  vertical axis in log scale to highlight the performance of the technique.}
\label{fig:c1_ideal}
\end{figure}

If, instead, we use the mask length selection approach currently implemented in IF, we obtain the performance shown in Figure \ref{fig:c1_actual}, which can be perfectly explained using the analysis made by Flandrin et al. in \cite{flandrin2007oneOrTwo}, that regards the actual extrema of the signal $s$ defined in \eqref{eq:discrete_signal_Rn}. In fact, when the amplitude $a$ of the LF component becomes big enough, in particular when $a\geq\frac{1}{f}$, the HF component contributes just to defining the local change in the concavity of the signal and no more to the local extrema of $s$. Therefore, unless properly guided, the IF method is no more able in this case to identify and separate two frequencies in the given signal. If we compare these last performance with the ones of EMD and Synchrosqueezing methods, we can see how IF algorithm, even with the currently implemented mask length selection approach, outperform both EMD and Synchrosqueezing in separating close by frequencies for values of the amplitude $a<\frac{1}{f}$. We observe that the Synchrosqueezing plot, reported in  \cite{wu2011oneOrTwo} and shown,  for readers convenience, in the right panel of Figure \ref{fig:c1_original}, does not include the range of amplitudes $[10^{0.4},\ 10^2]$.

\begin{figure}
\centering
\includegraphics[width=0.5\linewidth]{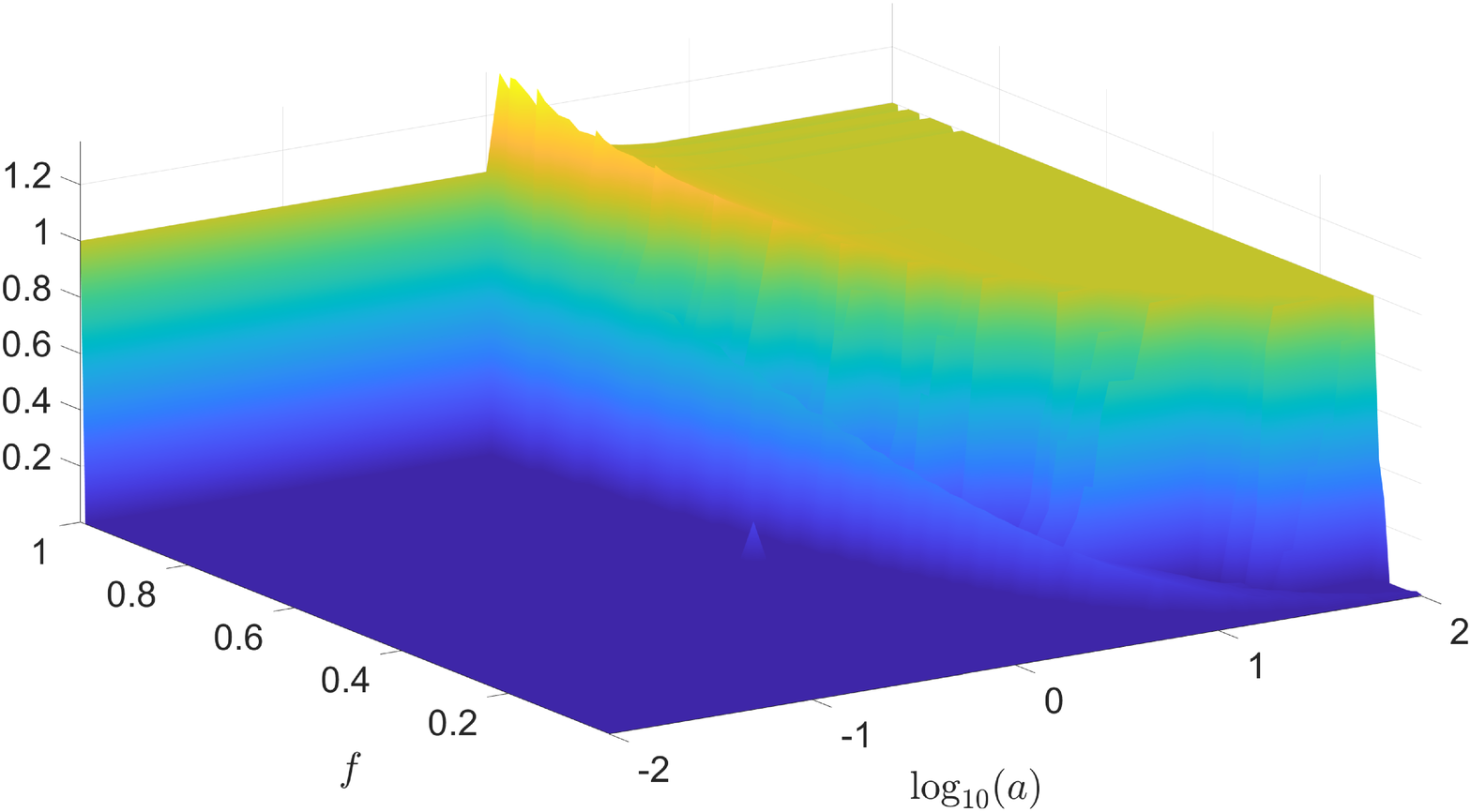}~\includegraphics[width=0.5\linewidth]{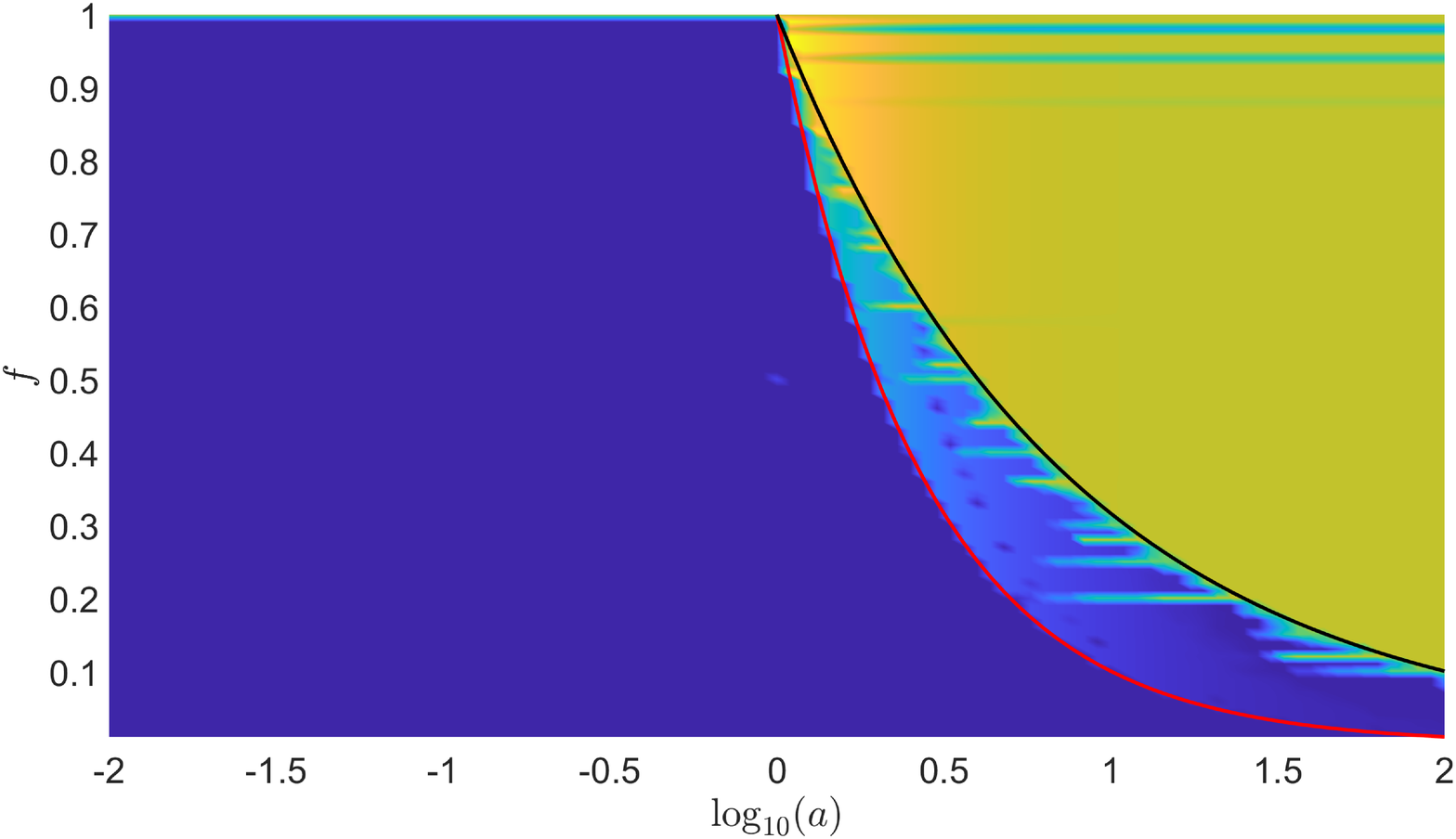}
\caption{Left panel, IF performance in separating two-tones with rational frequencies from the signal \eqref{eq:discrete_signal_Rn}, measured using the criterion \eqref{eq:c1} averaged over $\phi$, when the currently implemented mask length selection is used. Right panel, the projection onto the $(a,\ f)$-plane. The critical curves $af = 1$ and $af^2 = 1$ are plotted in solid red and black, respectively.}
\label{fig:c1_actual}
\end{figure}

\begin{figure}
\centering
\includegraphics[width=0.5\linewidth]{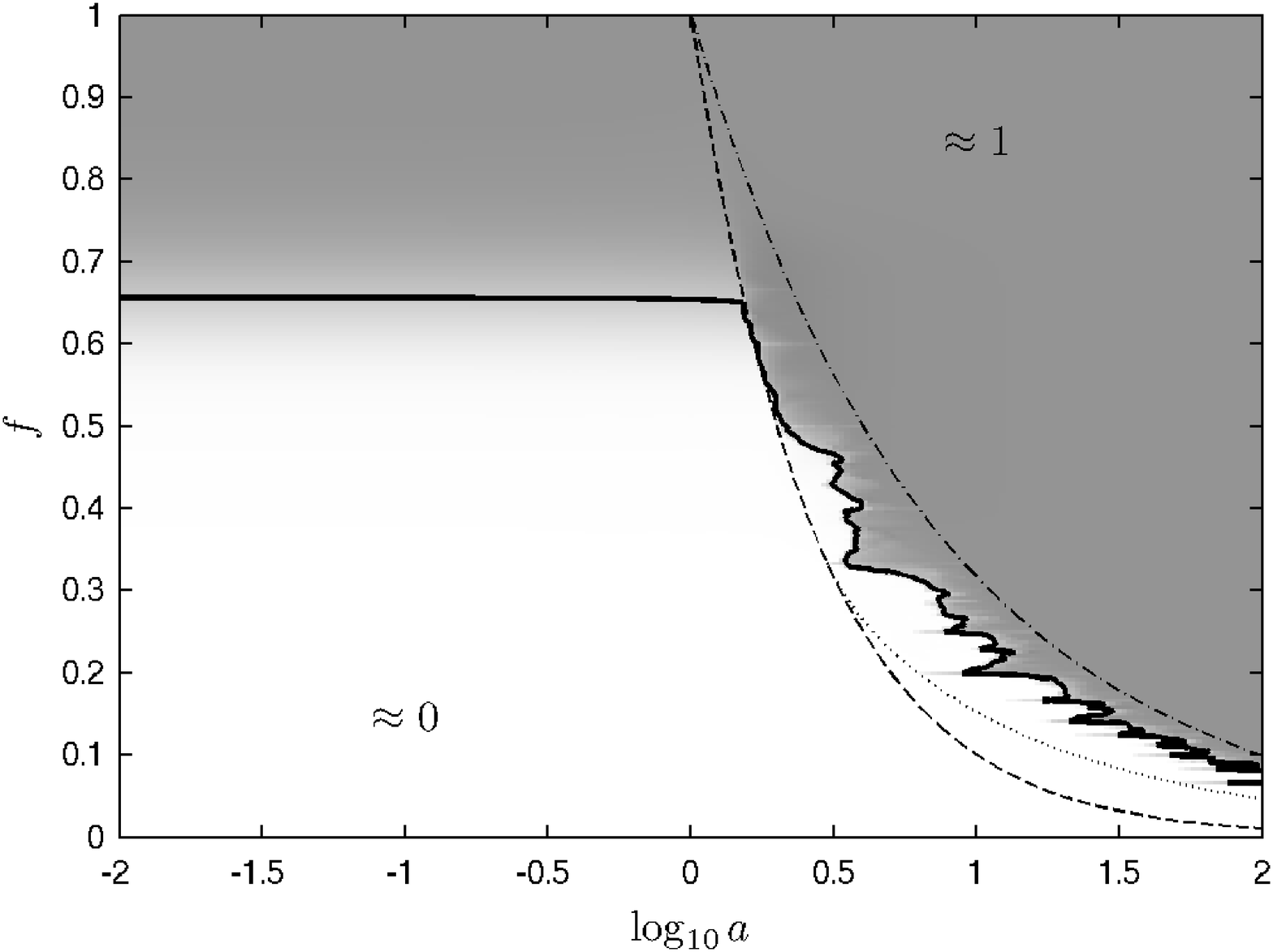}~\includegraphics[width=0.5\linewidth]{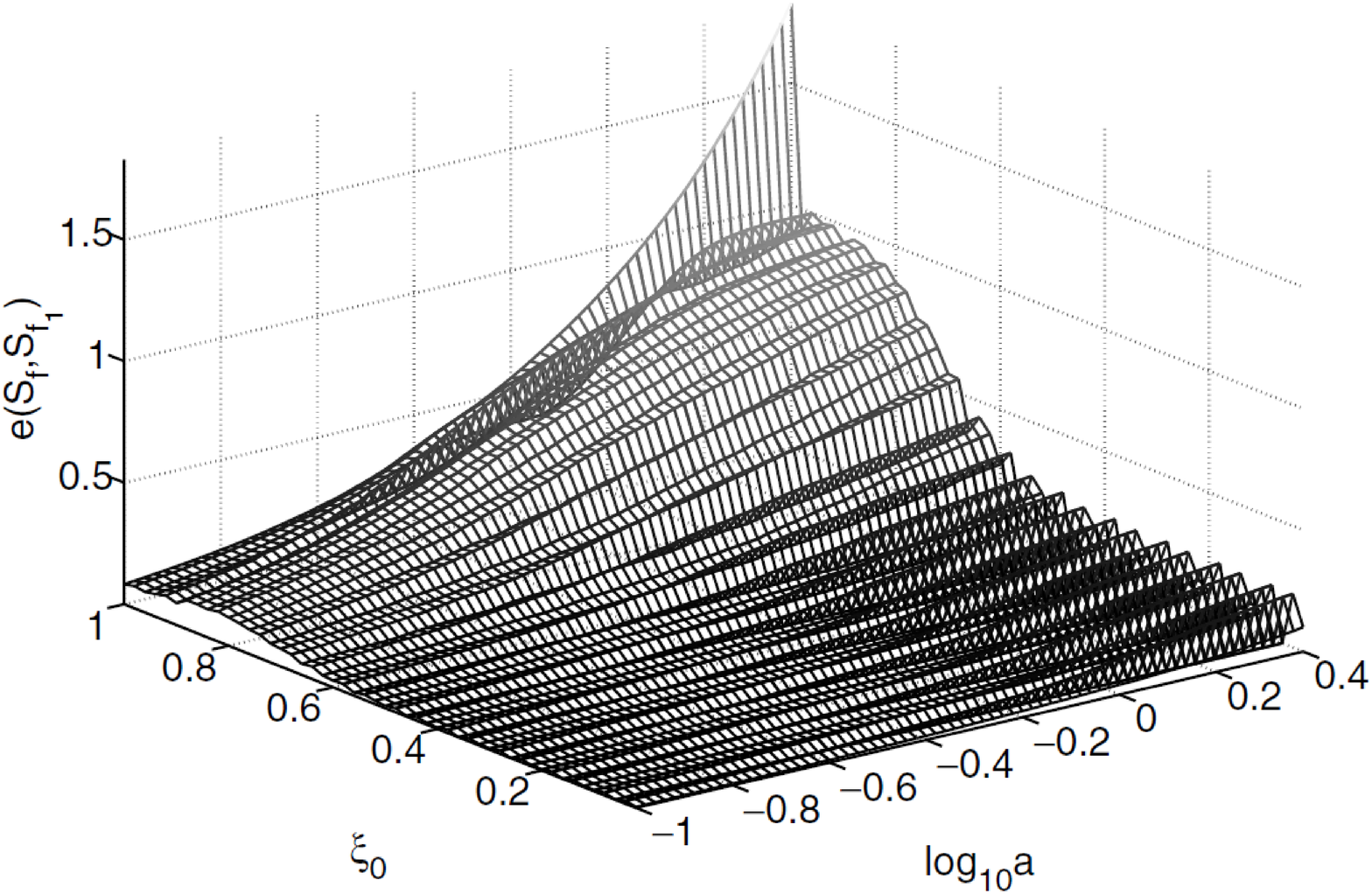}
\caption{EMD, left, and Synchrosqueezing, right, performance measure of separation for two-tones original images shown in \cite{flandrin2007oneOrTwo} and \cite{wu2011oneOrTwo}.}
\label{fig:c1_original}
\end{figure}

In order to improve the IF method mask length selection abilities, we can use the first, second or higher order derivatives of the signal. In doing so, we can enlarge the region in which the IF method, with its implemented mask length selection approach, can perfectly separate, up to machine precision, the two frequencies, Figure \ref{fig:c1_actual_2}.

This result is explained in the following theorem.

\begin{proposition}\label{pro:FlandrinExtended}
Given the signal $s$ defined in \eqref{eq:discrete_signal_Rn}, there exists $d\in\N$ big enough such that, by computing the DIF mask length based on the $d$--th derivative of $s$ signal, DIF algorithm can address the one or two frequencies question.
\end{proposition}
\begin{proof}
Given $s$ defined in \eqref{eq:discrete_signal_Rn}, for any $a\in\R$, there exists $d\in\N$ big enough such that $af^d\leq 1$ and the LF component of the $d$--th derivative of the signal $s$ will be comparable or negligible with respect to the HF component. Therefore any mask length selection process based on the extrema relative distance applied to $s^{(d)}$ will be able to identify the mask length required by DIF to separate properly the HF and LF components from the original signal $s$.
\end{proof}

\begin{figure}
\centering
\includegraphics[width=0.5\linewidth]{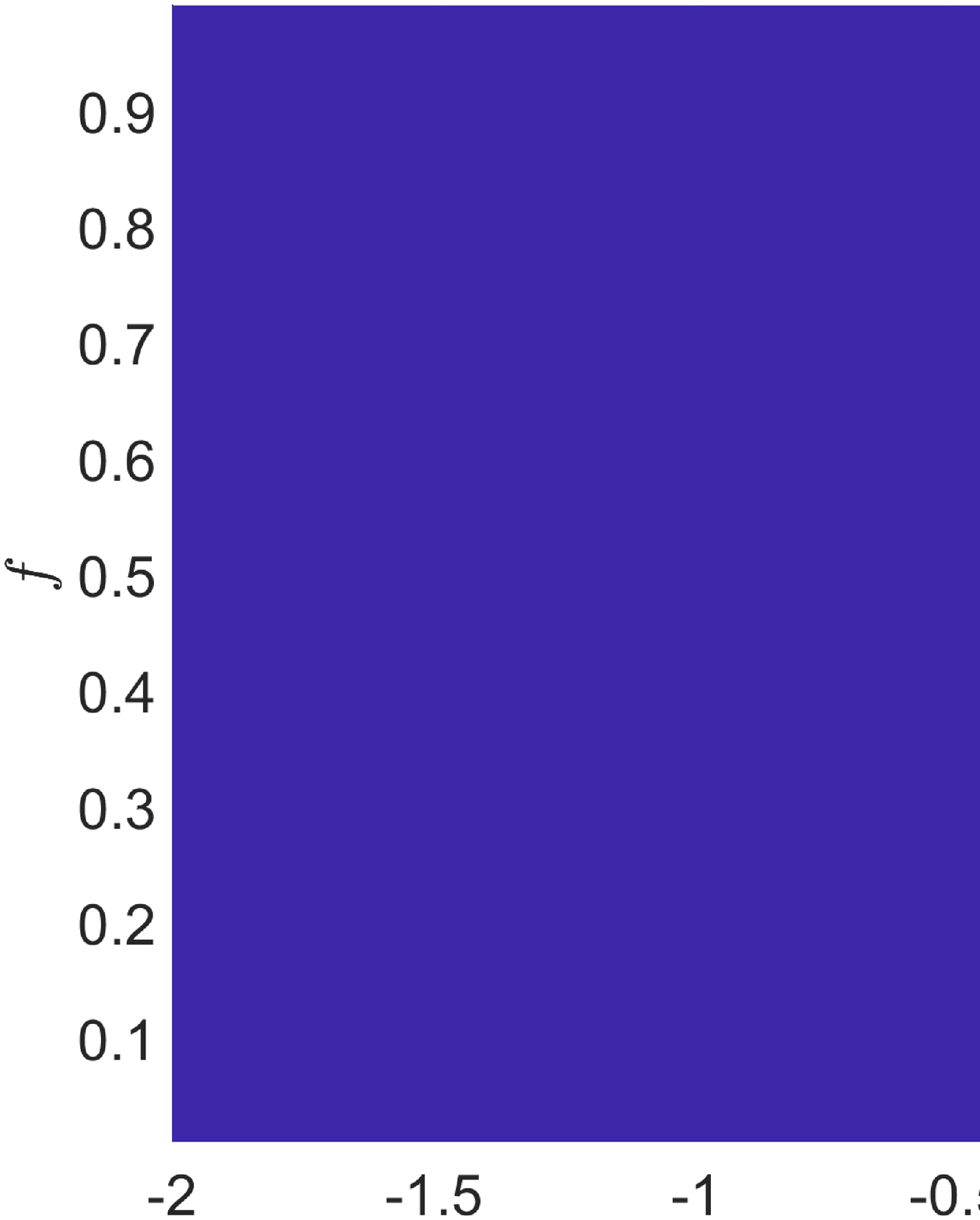}~\includegraphics[width=0.5\linewidth]{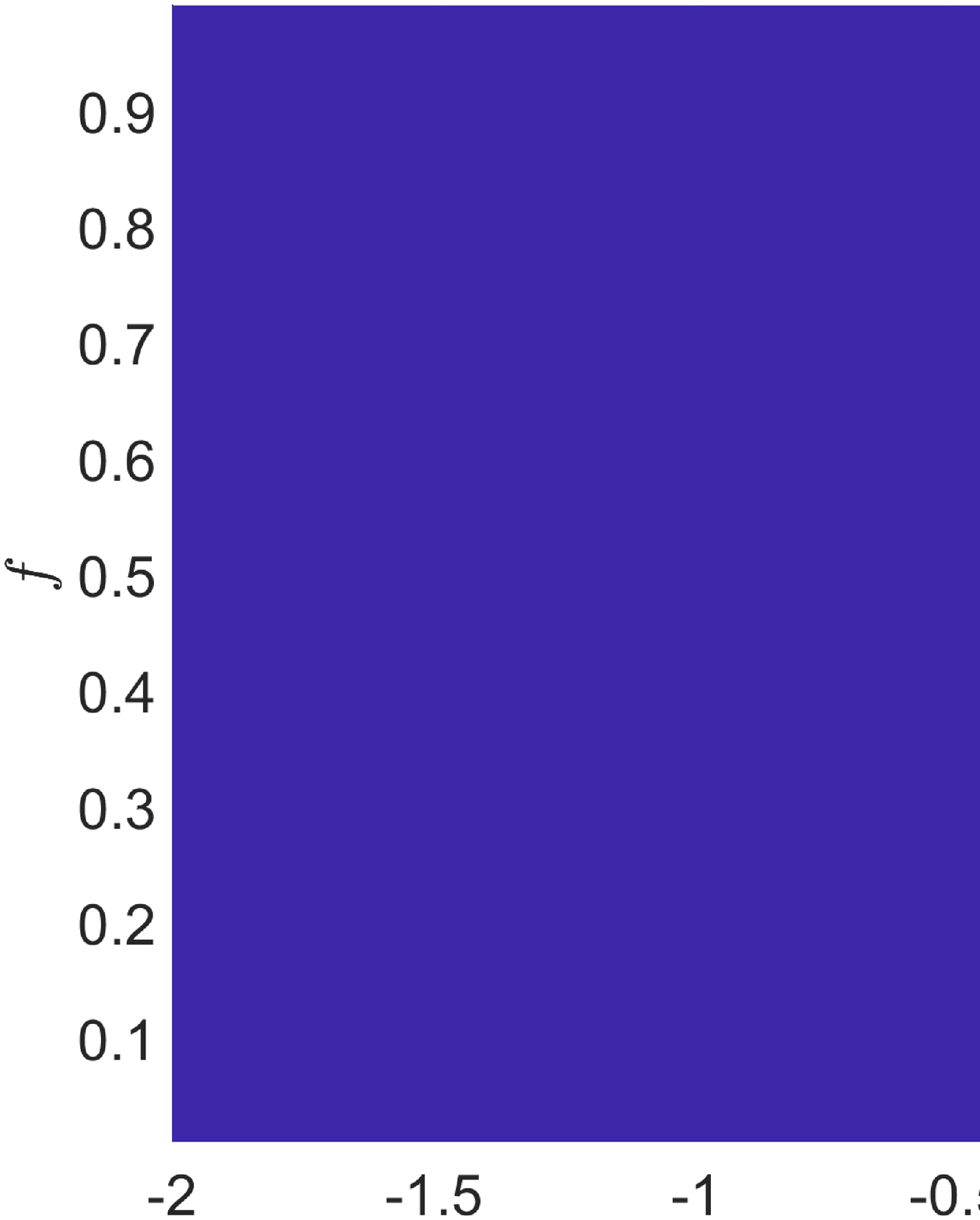}
\caption{IF performance in separating two-tones with rational frequencies from signal \eqref{eq:discrete_signal_Rn} when we use the first derivative (left panel) and the second derivative (right panel) of the signal to compute the mask length. The critical curves plotted are: $af = 1$ (red), $af^2 = 1$ (black), $af^3 = 1$ (green), $af^4 = 1$ (white).}
\label{fig:c1_actual_2}
\end{figure}

The last case we study in this section regards the separation of the two frequencies from signal \eqref{eq:discrete_signal_Rn} when $f$ is irrational. If we guide properly the IF method in its mask length selection we obtain remarkable results as shown in Figure \ref{fig:c1_irr_ideal}. If, instead, we use the currently implemented mask length selection approach, we obtain the performance shown in Figure \ref{fig:c1_irr_actual} which, again, fit well with the analysis made by Flandrin et al. in \cite{flandrin2007oneOrTwo} regarding the actual extrema of the signal $s$.

\begin{figure}
\centering
\includegraphics[width=0.5\linewidth]{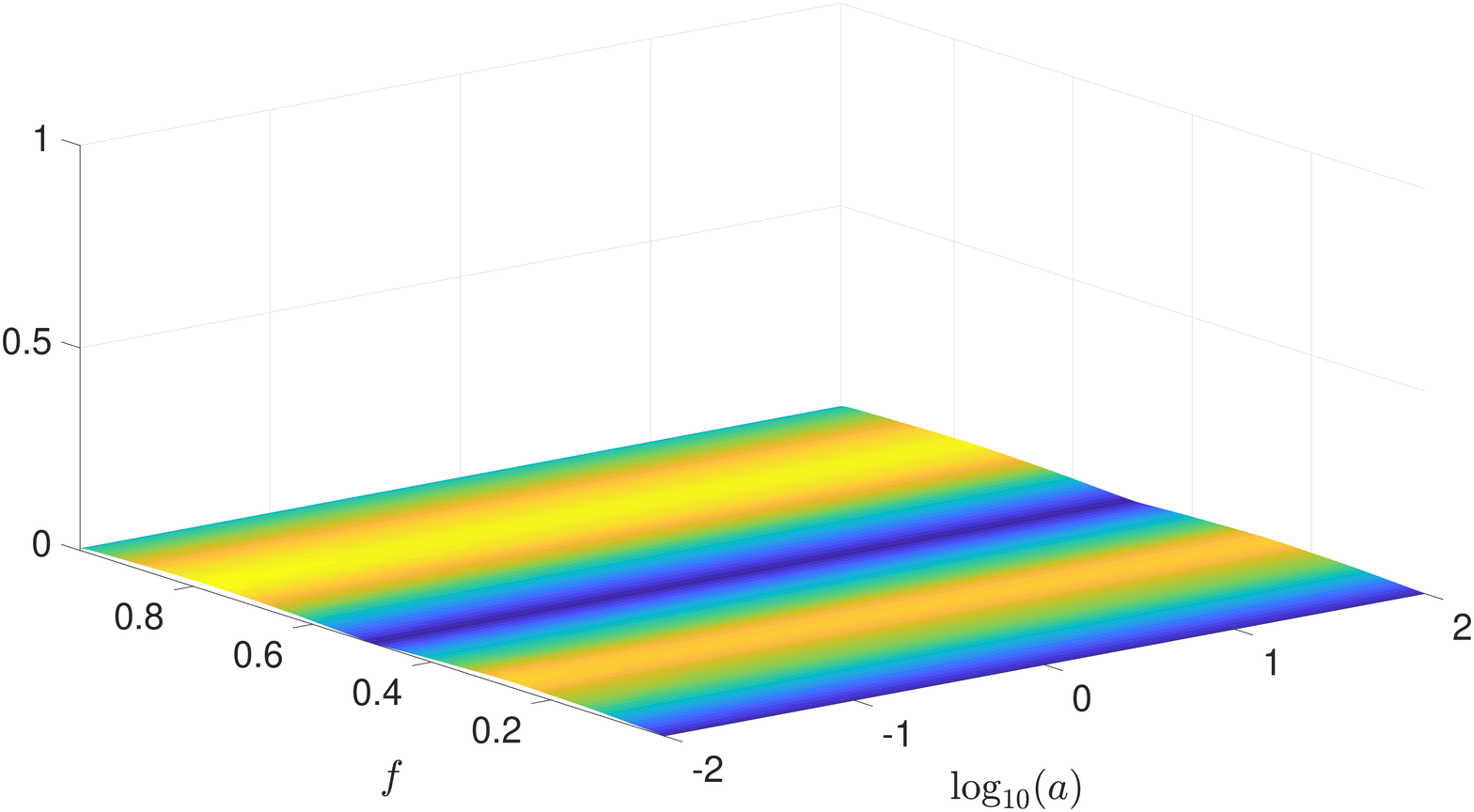}~\includegraphics[width=0.5\linewidth]{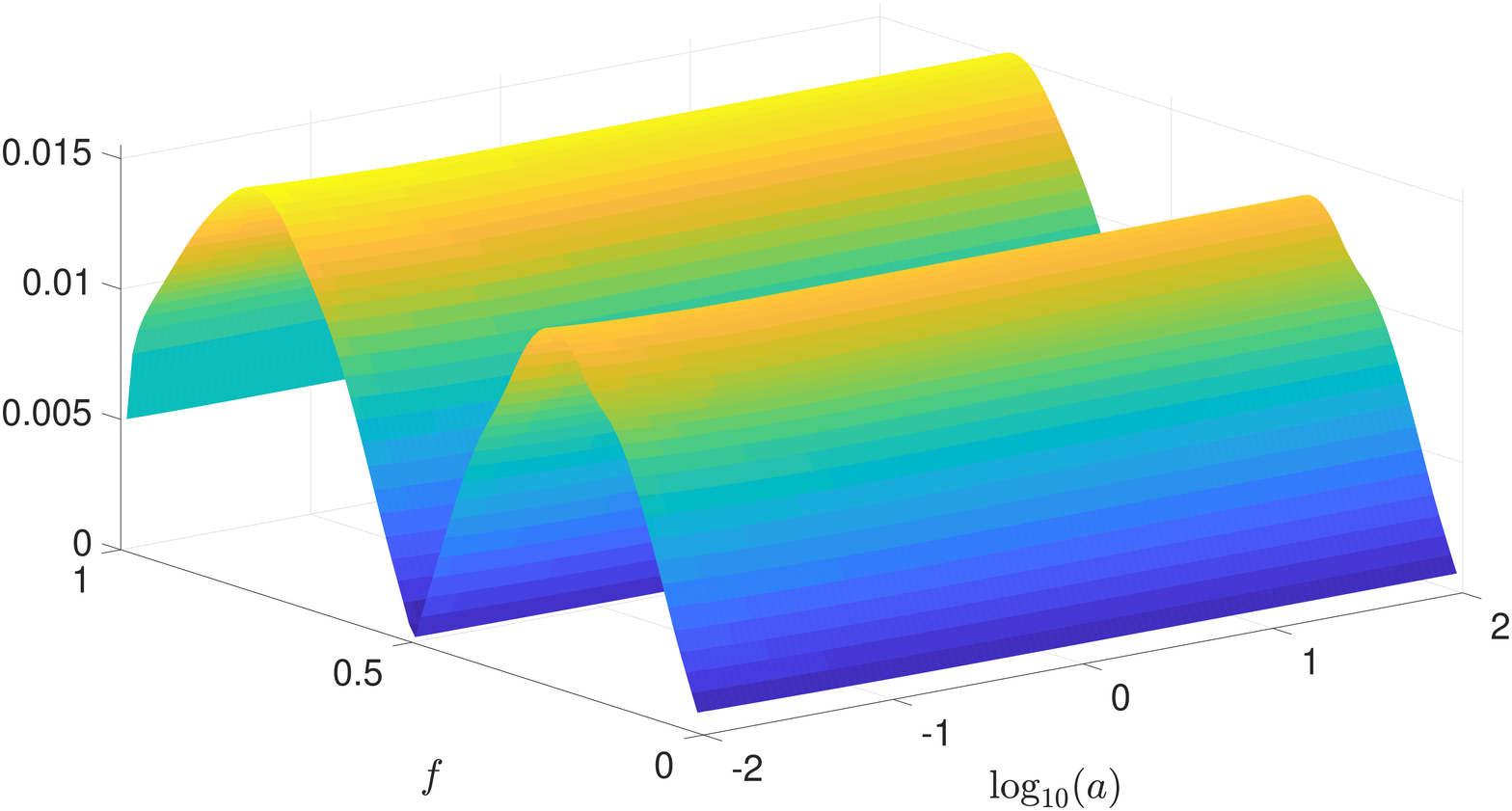}
\caption{IF performance in separating two-tones with irrational frequencies from signal \eqref{eq:discrete_signal_Rn}, when we assume an ideal way of selecting the mask length. We use criterion \eqref{eq:c1} averaged over $\phi$ to produce the plot shown in the left panel. In the right panel we present its zoomed in version.}
\label{fig:c1_irr_ideal}
\end{figure}

\begin{figure}
\centering
\includegraphics[width=0.5\linewidth]{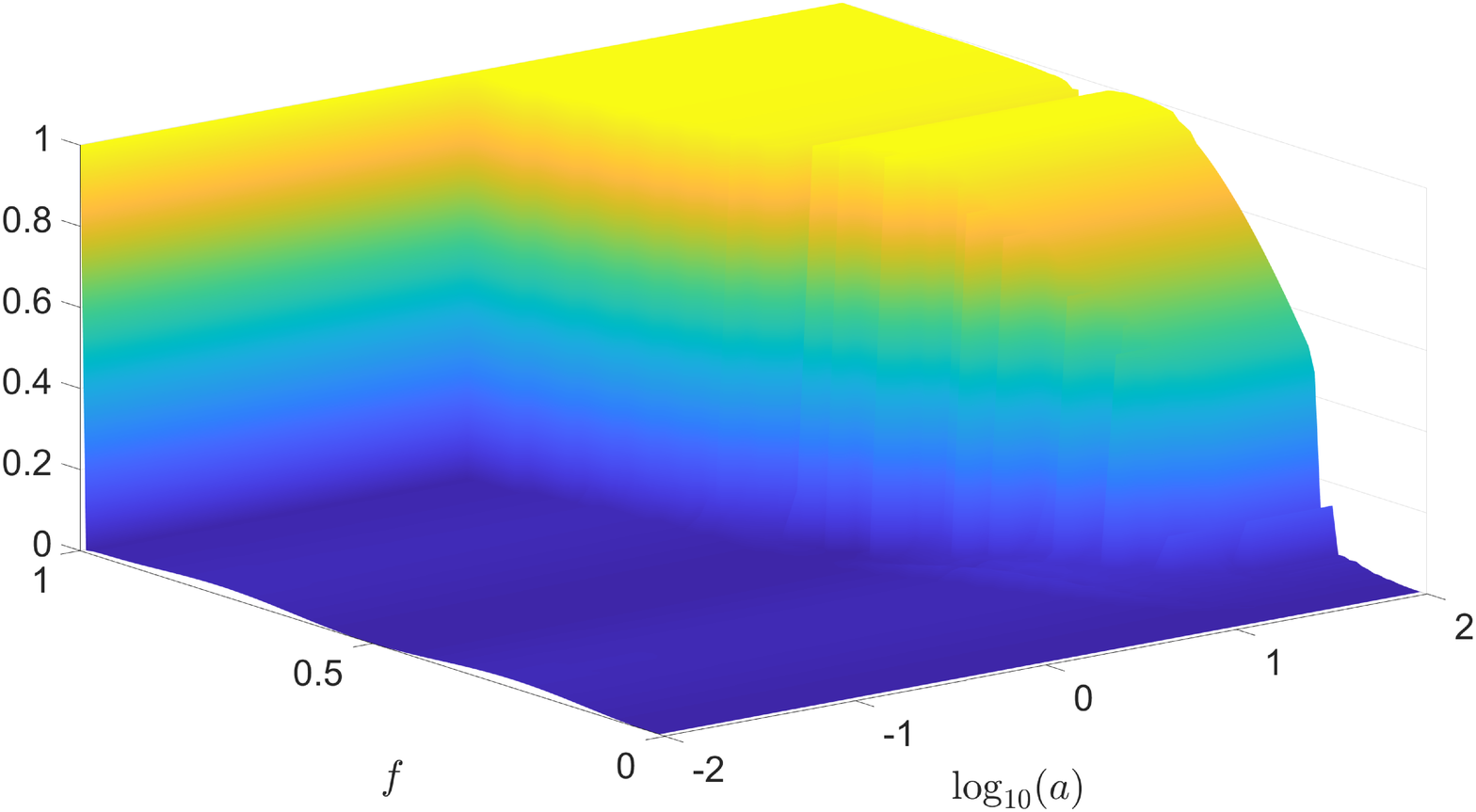}~\includegraphics[width=0.5\linewidth]{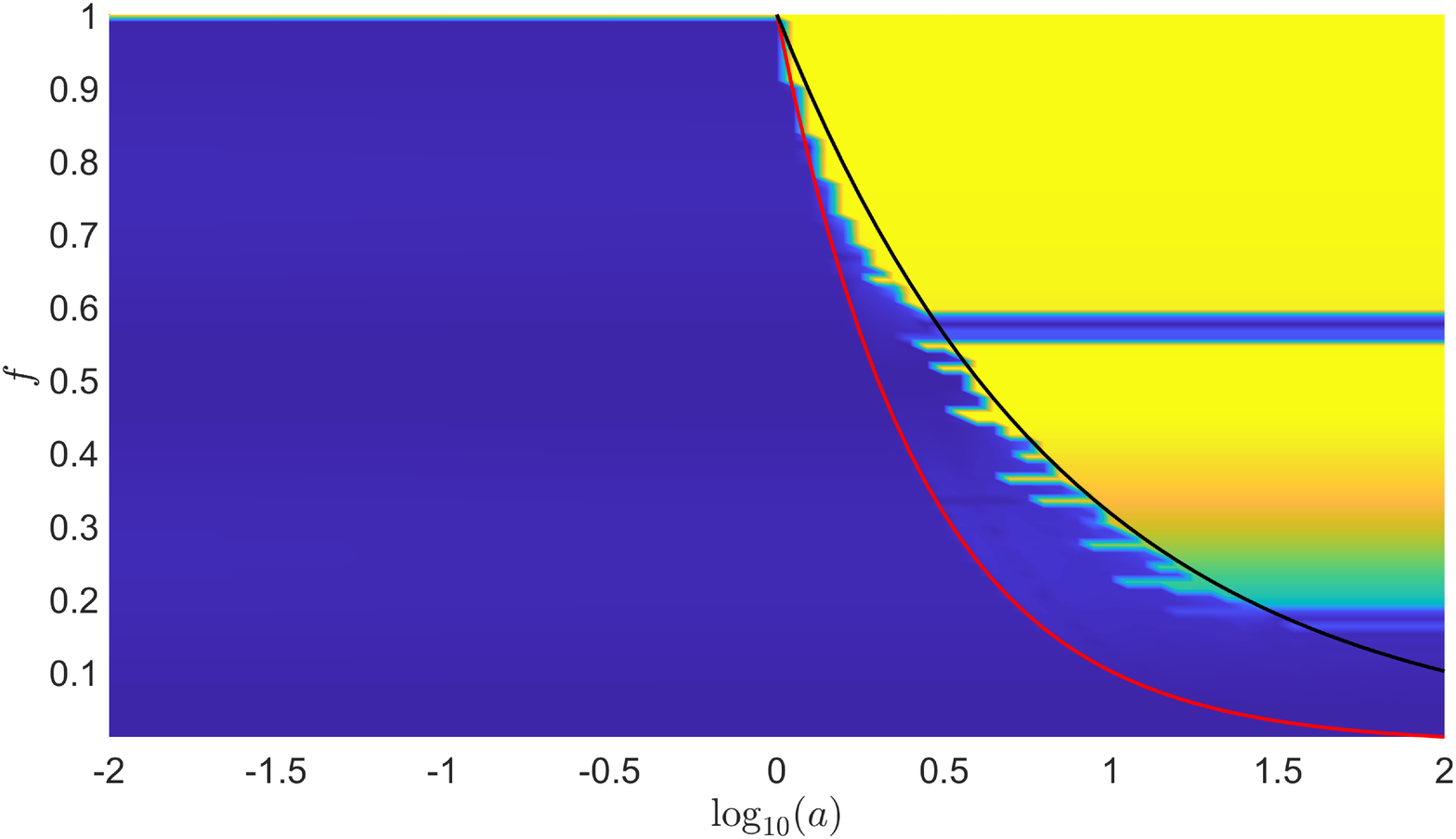}
\caption{Left panel, IF performance in separating two-tones with irrational frequencies from signal \eqref{eq:discrete_signal_Rn}, when the currently implemented approach for mask length selection is used. The results shown are obtained using criterion \eqref{eq:c1} when $\phi=3$. Right panel, the projection onto the $(a,\ f)$-plane. The critical curves $af = 1$ and $af^2 = 1$ are plotted in solid red and black, respectively.}
\label{fig:c1_irr_actual}
\end{figure}

\section{Conclusions}

The decomposition of non-stationary signals into simple oscillatory components, and their subsequent analysis, is an active research direction, which has a broad impact in many applied fields of research. Among the questions that researchers have been trying to address in recent years, there is the ``one or two frequencies'' problem, i.e. up to which extent a decomposition method is able to separate and extract two close-by stationary frequencies? This key question was originally raised and addressed by Rilling and Flandrin for the method called Empirical Mode Decomposition, in the paper \cite{flandrin2007oneOrTwo}. Few years later, Wu, Flandrin, and Daubechies raised and addressed the same question for the Synchrosqueezing algorithm \cite{wu2011oneOrTwo}.

In this work, given the recent theoretical results obtained in the analysis of the Iterative Filtering (IF) algorithm \cite{cicone2016adaptive,cicone2019Direct,cicone2021numerical}, and the impact that this method is having in many applied fields of research, see for instance \cite{coifman2017,sharma2017,mitiche2018,li2018entropy,stallone2020new} and references there in, we decided to address the ``one or two frequencies'' question also for the IF method.

Starting from the numerical analysis of the IF technique \cite{cicone2021numerical}, in this work we provided several new theoretical results. These new insights on the IF method properties allowed to address the question if the IF method can separate a signal into one or two frequencies, both in the ideal case of a continuously sampled signal and the more practical case of a discretely sampled one. In particular, we completed the numerical analysis presented in \cite{cicone2021numerical}, by studying also the case of a aperiodical discrete signal. Furthermore, we showed how to design filters, in the discrete setting, such that two important properties are guaranteed: the a priori convergence of the IF algorithm and, at the same time, the converge of IF to a non-zero simple oscillatory component, even if we iterate the decomposition steps for millions of times.

In the numerical section, we presented various examples that confirm the results obtained in the theoretical analysis of the algorithm. From these results, we were able to observe that the implemented IF algorithm either matches or outperforms the performance of both the Empirical Mode Decomposition and the Synchrosqueezing technique in separating two stationary frequencies.

Furthermore, we proposed the innovative idea of using the signal derivatives to increase the algorithm's ability to properly separate two frequencies. We provided theoretical explanations of why this approach can work, and we present numerical examples which confirm this claim.

At this point, the question that becomes natural to ask is whether, and up to which extent, the IF approach can separate two non-stationary components from a given signal. This is, to the best of our knowledge, a completely unexplored direction of research for any decomposition method proposed so far in the literature. We plan to tackle this question from the IF method perspective in future work.

Another open problem regards the proper computation of the IF filter support size. As shown in the numerical section, an ideal mask length selection method would allow the IF method to separate the two frequencies in any scenario. We leave this problem to future research.

\section*{Acknowledgments}
A. Cicone and S. Serra-Capizzano belong to the Units of the Italian Gruppo Nazionale di Calcolo Scientifico (GNCS) of the Istituto Nazionale di Alta Matematica (INdAM), Rome, Italy.
Their researches were partially supported by the GNCS. H. Zhou work was partially supported by research grants NSF DMS-1830225 and ONR N00014-21-1-2891.


\begin{thebibliography}{99}

\bibitem{flandrin1998}
P. Flandrin: Time-frequency/time-scale analysis, Academic press, 1998.

\bibitem{lin2009iterative}
{\textsc L.~Lin, Y.~Wang, H.~Zhou},
{\em Iterative filtering as an alternative algorithm for empirical mode decomposition}, Advances in Adaptive Data Analysis, 1(4):543--560, 2009.

\bibitem{B.Boashash2003}
{\textsc B. Boashash},
{\em Time-Frequency Signal Analysis and Processing: A Comprehensive Reference}, Elsevier, 2003.

\bibitem{cicone2019Direct}
{\textsc A. Cicone},
{\em  Iterative Filtering as a direct method for the decomposition of nonstationary signals},
Numerical Algorithms, (2020), pp. 1--17.

\bibitem{cicone2019spectral}
{\textsc A. Cicone, C. Garoni, S. Serra-Capizzano},
{\em Spectral and convergence analysis of the Discrete ALIF method},
Linear Algebra and its Applications, 580 (2019), pp. 62--95.

\bibitem{cicone2016adaptive}
{\textsc A.~Cicone, J.~Liu, H.~Zhou},
{\em Adaptive local iterative filtering for signal decomposition and instantaneous frequency analysis},
Appl. Comput. Harmon. Anal., 41(2):384--411, 2016.

\bibitem{cicone2021numerical}
{\textsc A. Cicone, H. Zhou},
{\em Numerical Analysis for Iterative Filtering with New Efficient Implementations Based on FFT},
Numerische Mathematik,  147, 1 (2021), pp. 1--28.

\bibitem{cicone2019BC}
{\textsc A. Cicone, P. Dell'Acqua},
{\em Study of boundary conditions in the Iterative Filtering method for the decomposition of nonstationary signals},
Journal of Computational and Applied Mathematics, Volume 373, 112248, 2019.

%% EMD algorithms

\bibitem{huang1998empirical}
{\textsc N.E. Huang, Z.~Shen, S.R. Long, M.C. Wu, H.H. Shih, Q.~Zheng, N.C. Yen, C.C. Tung, H.H. Liu},
{\em The empirical mode decomposition and the hilbert spectrum for nonlinear and non-stationary time series analysis},
Proceedings of the Royal Society of London. Series A: Mathematical, Physical and Engineering Sciences, 454(1971):903, 1998.

%% Other papers

\bibitem{auger1995improving}
{\textsc F. Auger, P. Flandrin},
{\em Improving the readability of time-frequency and time-scale representations by the reassignment method},
IEEE Transactions on signal processing, 43(5):1068--1089, 1995.

\bibitem{flandrin2007oneOrTwo}
{\textsc G. Rilling, P. Flandrin},
{\em One or two frequencies? The empirical mode decomposition answers.},
IEEE transactions on signal processing, 56(1): 85--95, 2007.

\bibitem{wu2011oneOrTwo}
{\textsc H.T. Wu, P. Flandrin, I. Daubechies},
{\em One or two frequencies? The synchrosqueezing answers},
Advances in Adaptive Data Analysis, 3(01n02): 29--39, 2011.

\bibitem{daubechies2011synchrosqueezed}
{\textsc I. Daubechies, J. Lu, H.T. Wu},
{\em Synchrosqueezed wavelet transforms: An empirical mode decomposition-like tool},
Applied and computational harmonic analysis, 30(2): 243--261, 2011.

% EMD Applications

\bibitem{lei2013review}
{\textsc Y. Lei, J. Lin, Z. He, M.J. Zuo},
{\em A review on empirical mode decomposition in fault diagnosis of rotating machinery}, Mechanical systems and signal processing, 35(1-2): 108-126, 2013.

\bibitem{boudraa2007noise}
{\textsc A.O. Boudraa, J.C. Cexus, S. Benramdane, A. Beghdadi},
{\em Noise filtering using empirical mode decomposition}, 9th International Symposium on Signal Processing and Its Applications, 1-4, 2007

\bibitem{echeverria2001application}
{\textsc J.C. Echeverria, J.A. Crowe, M.S. Woolfson, B.R. Hayes-Gill},
{\em Application of empirical mode decomposition to heart rate variability analysis},
Medical and Biological Engineering and Computing, 39(4), 471-479, 2001.

\bibitem{stallone2020new}
{\textsc A. Stallone, A. Cicone, M. Materassi},
{\em New insights and best practices for the successful use of Empirical Mode Decomposition, Iterative Filtering and derived algorithms},
Scientific reports, 10(1), 1-15, 2020.

% FIF applications

\bibitem{li2018entropy}
{\textsc Y. Li, X. Wang, Z. Liu, X. Liang, S. Si},
{\em The entropy algorithm and its variants in the fault diagnosis of rotating machinery: A review},
Ieee Access, 6:66723-66741, 2018.

\bibitem{sharma2017}
{\textsc R. Sharma, R.B. Pachori, A. Upadhyay},
{\em Automatic sleep stages classification based on iterative filtering of electroencephalogram signals},
Neural Computing and Applications, 28(10): 2959-2978, 2017

\bibitem{coifman2017}
{\textsc R.R. Coifman, S. Steinerberger, H.T. Wu},
{\em Carrier frequencies, holomorphy, and unwinding}.
SIAM Journal on Mathematical Analysis, 49(6): 4838-4864, 2017

\bibitem{mitiche2018}
{\textsc I. Mitiche, G. Morison, A. Nesbitt, M. Hughes-Narborough, B.G. Stewart, P. Boreham},
{\em Classification of partial discharge signals by combining adaptive local iterative filtering and entropy features},
Sensors, 18(2): 406, 2018

% EMD math analysis

\bibitem{ge2018theoretical}
{\textsc H. Ge, G. Chen, H. Yu, H. Chen, F. An},
{\em Theoretical analysis of empirical mode decomposition},
Symmetry, 10(11), 623, 2018.

\bibitem{huang2014introduction}
{\textsc N.~E. Huang},
{\em Introduction to the hilbert--huang transform and its related  mathematical problems},
Hilbert--Huang transform and its applications, 1--26, 2014.

\bibitem{huang2009convergence}
{\textsc C.~Huang, L.~Yang, Y.~Wang},
{\em Convergence of a convolution-filtering-based algorithm for empirical mode decomposition},
Advances in Adaptive Data Analysis, 1(04):561--571, 2009.

\end{thebibliography}
\end{document}